\documentclass[preprint,12pt]{elsarticle}

\usepackage{amssymb}
\usepackage{amsthm}
\usepackage{amsmath}

\usepackage{xcolor}

\theoremstyle{definition}
\newtheorem{definition}{Definition}[section]
\newtheorem{proposition}{Proposition}[section]
\newtheorem{theorem}{Theorem}[section]
\newtheorem{corollary}{Corollary}[section]
\newtheorem{lemma}{Lemma}[section]
\newtheorem{result}{Result}[section]
\newtheorem{remark}{Remark}
\newtheorem{example}{Example}[section]

\newtheorem*{remark*}{Remark}

\newtheorem*{theorem*}{Theorem}

\newcommand{\A}{\mathcal{A}}

\begin{document}
	
	\begin{frontmatter}
		
		
		
		\title{New Tight Wavelet Frame Constructions Sharing Responsibility}
        
		
		\author[inst1]{Youngmi Hur}
		\ead{yhur@yonsei.ac.kr}
		
		\affiliation[inst1]{organization={Department of Mathematics, Yonsei University},
			city={Seoul},
			postcode={03722}, 
			country={Korea}}
		
		\author[inst2]{Hyojae Lim}
		\ead{hyojae.lim@oeaw.ac.at}
		
		\affiliation[inst2]{organization={Johann Radon Institute for Computational and Applied Mathematics~(RICAM)},
			postcode={4040}, 
			city={Linz},
			country={Austria}}
		
		\begin{abstract}
			Tight wavelet frames (TWFs) in \(L^2(\mathbb{R}^n)\) are versatile, and are practically useful due to their perfect reconstruction property. Nevertheless, existing TWF construction methods exhibit limitations, including a lack of specific methods for generating mother wavelets in extension-based construction, and the necessity to address the sum of squares (SOS) problem even when specific methods for generating mother wavelets are provided in SOS-based construction. Many TWF constructions begin with a given refinable function. However, this approach places the entire burden on finding suitable mother wavelets. In this paper, we introduce TWF construction methods that spread the burden between both types of functions: refinable functions and mother wavelets. These construction methods offer an alternative approach to addressing the SOS problem. We present examples to illustrate our construction methods.
		\end{abstract}
		
		%
		
		
	\end{frontmatter}
	
	
	
	\section{Introduction}
	
	Orthonormal wavelets in \(L^2(\mathbb{R}^n)\) are well-known mathematical tools for analyzing various signals. These wavelet systems possess the perfect reconstruction property, ensuring minimal loss of information during the processes of decomposition and reconstruction due to their orthogonality. Wavelet frames offer more practical and flexible structures compared to orthonormal wavelets, while still providing the perfect reconstruction property. In particular, tight wavelet frames (TWFs) in \(L^2(\mathbb{R}^n)\) are advantageous for applications, as their duals are (essentially) the same as the primal ones, which makes them a natural generalization of orthonormal wavelets.
	
	A fundamental approach for obtaining a TWF is using the unitary extension principle (UEP) \cite{[RS]UEP}. The UEP condition (c.f., (\ref{eqn : UEP_condition})) is sufficient for a collection of mother wavelets
	to generate a TWF with a given refinable function with its refinement mask.
	While constructions employing the unitary extension principle are simple and easy to handle, the vanishing moments of the TWF are often not optimal. To address this issue, constructions utilizing the oblique extension principle (OEP) have been introduced \cite{[ChuiStockler]OEP, [DHRS]Framelets}. The OEP is a generalization of the UEP and can potentially enhance the vanishing moments of the TWF with the help of the so-called vanishing-moment recovery function. Both of these extension-based constructions allow us to generate a TWF when the refinable function is given together with a collection of mother wavelets satisfying corresponding conditions. Although these constructions have been used extensively for obtaining TWFs, one of the main difficulties in using them is that they still lack a specific procedure for generating the mother wavelets.
	
	Subsequently, constructions based on a sum of the (Hermitian) squares (SOS) representation provide a new way to generate specific mother wavelets constituting a TWF, when a refinable function (with its refinement mask) is given. One such construction works under the sub-QMF condition, which is the nonnegativity of the function \(f_\tau\) detailed in (\ref{eqn : sub-QMF}) within Section~\ref{subsection : SOS-based construction},
	where \(\tau\) is a refinement mask~\cite{CPSS, LaiStockler}. These papers show that \(f_\tau\) has a sum of squares representation if and only if there exist mother wavelets that satisfy the unitary extension principle conditions. Furthermore, this construction method gives a specific form of mother wavelets that generate a TWF when a sum of squares representation of \(f_\tau\) is available. 
	
	A generalization to the above construction is also introduced, and it gives an SOS-based construction that corresponds to the oblique extension principle condition, which works under the oblique sub-QMF condition (c.f.,~(\ref{eqn : oblique sub-QMF})) as shown in~\cite{[HLO]MultiTWFhighVM}. The oblique sub-QMF condition corresponds to the nonnegativity of the function derived from a vanishing-moment recovery function and a refinement mask. Similar to the UEP case, the existence of SOS generators of this nonnegative function and the existence of mother wavelets satisfying the oblique extension principle conditions are closely related. This method likewise provides an explicit form of mother wavelets that generate a TWF when SOS generators are found.
	
	However, there is a significant difficulty in constructions based on a sum of squares representation. This is essentially due to the nature of the SOS problem itself. More precisely, determining whether a nonnegative polynomial admits an SOS form, often referred to as the {\em SOS problem}, has long been recognized as a major issue in real algebraic geometry. It is well known that not every nonnegative {\em algebraic polynomial} can be expressed as a sum of squares of algebraic polynomials \cite{Hilbert1888,benoist2017}. This question evolved into whether every nonnegative algebraic polynomial can be represented as an SOS of {\em rational algebraic polynomials}, a problem famously known as Hilbert's 17th problem \cite{RudinSOSPoly}. Today it is established that any such nonnegative polynomial can admit an SOS representation using rational algebraic polynomials \cite{Artin1927,Pfister1967,benoist2017}. 
	
	As we will present in more detail in Section~\ref{subsection : SOS-based construction}, the particular challenges for TWF constructions arise in the context of {\em trigonometric polynomials} (and their {\rm rational} counterparts), rather than algebraic ones. Various studies have investigated whether a nonnegative trigonometric polynomial can be written as the sum of squares of either trigonometric polynomials or {\em rational} trigonometric polynomials \cite{CPSS, CPSSII, Scheiderer2006,[HurZach]LAA, [HurZach]Interpolatory}. Despite these efforts, using an SOS representation to construct TWFs remains demanding, largely because the functions forming the sum of squares must be found explicitly.

	Moreover, many of the previous SOS-based construction methods focus on identifying wavelet masks (and consequently mother wavelets) that meet conditions derived from the refinement mask, assuming the refinement mask is already known. This method largely burdens the wavelet masks, necessitating solving the SOS problem.
	
	\medskip
	
	Motivated by the straightforward construction approach detailed in~\cite{[Hur]TWFPD} and after thoroughly examining the method presented therein, we have found that tight wavelet frames can be efficiently constructed using these SOS-based methods without the need to directly solve the SOS problem. We must satisfy the same identity (c.f., (\ref{eqn : sub-QMF SOSv2}) or (\ref{eqn : oblique sub-QMF SOSv2})); however, our method offers the advantage of not being confined to a single strategy for each of these identities. Specifically, in satisfying the identity in (\ref{eqn : sub-QMF SOSv2}), whereas previous approaches calculate the left-hand side of the identity from the given refinement mask and seek corresponding wavelet masks, our approach introduces conditions (see~(\ref{eqn : our condition on p,g for subQMF})) that ensure the identity in (\ref{eqn : sub-QMF SOSv2}) is upheld.
	This flexibility allows for a more balanced distribution of responsibility among the masks involved (namely, the wavelet masks and the refinement mask), leading to a more adaptable construction process. However, a potential drawback of our method is that it might yield a refinable function that is less conventional in appearance. A more detailed comparison of our construction method with previous ones is provided in Section~\ref{section : our new construction}, immediately after the proof of Theorem~\ref{theorem: our subQMF construction}.
	
	Furthermore, given that our construction approach does not start with a pre-specified refinable function, it becomes crucial to verify whether the refinable function, emerging from the constructed refinement mask, is well-defined in \(L^2(\mathbb{R}^n)\) and exhibits the necessary convergence property. This issue is well-addressed in existing literature for the case when the refinement mask is a trigonometric polynomial \cite{BHan}. However, to our knowledge, no comparable results have been established for refinement masks that are rational trigonometric polynomials. Hence, in this paper, we present such results.
	
	\medskip
	
	The rest of our paper is structured as follows. In Section~\ref{section : Preliminaries}, we offer a brief introduction to tight wavelet frames and discuss construction methods based on extension principles and sum of squares techniques. We present our main results concerning the construction of TWFs sharing responsibility, along with examples illustrating our findings, in Section~\ref{section : our new construction}. Section~\ref{section : refinable ftn with rationalTP} articulates sufficient conditions that ensure the refinable function's well-definedness and the convergence property when employing rational trigonometric polynomials as refinement masks. The conclusion of our work is summarized in Section~\ref{section : conclusion}, while some technical details are provided in the Appendix.
	
	\section{Preliminaries}
	\label{section : Preliminaries}
	
	\subsection{Tight wavelet frames and the extension principles}
	\label{subsection : TWF and EP}
	
	In the definition below and for the rest of the paper, the notation \(\lambda\) is used to denote a dilation factor, an integer greater than or equal to~\(2\).
	\begin{definition}
		\begin{enumerate}[(a)]
			\item
			Let 
			\(\Psi \subset L^2(\mathbb{R}^n)\) be a finite set of functions. 
			A \textit{wavelet system} \(X(\Psi)\) generated by \(\Psi\) is a collection of dilated and translated functions of each function in \(\Psi\), i.e., 
			\(X(\Psi) := \lbrace \psi_{j,k} := \lambda^{jn/2} \psi (\lambda^j \cdot - k) : \psi \in \Psi, j\in\mathbb{Z}, k \in \mathbb{Z}^n \rbrace\). In this case, each function \(\psi\in\Psi\) is called a \textit{mother wavelet}.
			\item 
			A wavelet system  \(X(\Psi)\) is a \textit{tight frame} if  \(\sum_{\psi\in\Psi}\sum_{j,k} \lvert \langle f, \psi_{j,k} \rangle \rvert^2 = \lVert f \rVert_{L^2(\mathbb{R}^n)}^2\), for all \(f \in L^2(\mathbb{R}^n)\). 
		\end{enumerate}
	\end{definition}
	
	It is easy to see that \(X(\Psi)\) is a tight wavelet frame (TWF) if and only if it has a perfect reconstruction property: \(f = \sum_{\psi\in\Psi}\sum_{j,k} \langle f, \psi_{j,k} \rangle \psi_{j,k}\), for all \(f \in L^2(\mathbb{R}^n)\). 
	
	\begin{definition}
		A function \(\phi\in L^2(\mathbb{R}^n)\) is called \textit{refinable} if \(\widehat{\phi}(\lambda \cdot) = \tau(\cdot) \widehat{\phi}(\cdot)\), where \(\tau\) is a \(2\pi\)-periodic function known as the associated \textit{refinement mask}. 
	\end{definition}
	
	Throughout the paper, we normalize the Fourier transform as \(\widehat{f}(\xi) = \int_{\mathbb{R}^n} f(x) e^{-ix \cdot \xi} dx\) for \(f\in L^1(\mathbb{R}^n) \cap L^2(\mathbb{R}^n)\).
	Also, we will assume that \(\tau\) is either a trigonometric polynomial or a rational trigonometric polynomial satisfying \(\tau(0)=1\) in this paper. Specifically, a rational trigonometric polynomial \(\tau\) is defined as \(\tau= \tau_1/\tau_2\), where \(\tau_1\) and \(\tau_2\) are trigonometric polynomials with \(\tau_2 \not\equiv0\), and the greatest common divisor of \(\tau_1\) and \(\tau_2\) is a constant.

	The following is a construction method called the unitary extension principle \cite{[RS]UEP}, which is tailored to our specific context. Here and below, we adopt the notation \(\Gamma :=  (2\pi/\lambda) \cdot \lbrace0, \cdots, \lambda-1 \rbrace^n\).

    Before presenting the precise statement, we remark that the unitary extension principle (UEP) remains valid even when refinement masks \(\tau\) and the functions \(q_l\) are rational trigonometric polynomials. However, for simplicity of presentation and to ensure compact support, we restrict ourselves throughout this paper to the case of trigonometric polynomial masks.
	
	\begin{result}[\cite{[RS]UEP}]
		Let \(\phi \in L^2(\mathbb{R}^n)\) be a refinable function associated with a refinement mask \(\tau\) that is a trigonometric polynomial.
		Suppose that the trigonometric polynomials \(q_l, 1 \leq l \leq L\), satisfy for all \(\xi\in\mathbb{T}^n\)
		\begin{equation}
			\tau(\xi) \overline{\tau(\xi + \gamma)} + \sum_{l=1}^L q_l(\xi) \overline{q_l(\xi + \gamma)} = \begin{cases}
				1, & \text{if } \gamma = 0,\\
				0, & \text{if } \gamma\in \Gamma \setminus \{0\} .
			\end{cases}
			\label{eqn : UEP_condition}
		\end{equation}
		For each \(1 \leq l \leq L\), if we define \(\psi^{(l)} \in L^2(\mathbb{R}^n)\) via \(\widehat{\psi^{(l)}}(\lambda \cdot) = q_l(\cdot) \widehat{\phi}(\cdot)\), then the wavelet system \(X(\psi^{(1)}, \cdots, \psi^{(L)})\) is a compactly supported TWF for \(L^2(\mathbb{R}^n)\).
		\label{result : UEP}
	\end{result}
	
	The following outlines a construction method termed the oblique extension principle \cite{[DHRS]Framelets}, specialized to our setting. We use the notation \(\sigma(\phi)\) to represent the set \(\{\xi \in \mathbb{T}^n : \widehat{\phi}(\xi + 2 \pi k) \not=~0 \text{ for some } k \in \mathbb{Z}^n\}\), which is determined up to a null set.
	\begin{result}[\cite{[DHRS]Framelets}]
		Let \(\phi\in L^2(\mathbb{R}^n)\) be a refinable function associated with a refinement mask \(\tau\) that is a rational trigonometric polynomial. Let \(q_1, \cdots, q_L\) be rational trigonometric polynomials. 
		Suppose that 
		\begin{enumerate}[(a)]
			\item Each function \(\tau, q_1, \cdots, q_L\) belongs to \(L^\infty(\mathbb{T}^n)\).
			\item The refinable function \(\phi\) satisfies \(\lim_{\xi\to0} \widehat{\phi}(\xi) = 1\).
			\item The function \([\widehat{\phi}, \widehat{\phi}](\cdot):= \sum_{k \in \mathbb{Z}^n} \lvert \widehat{\phi}(\cdot + 2\pi k) \rvert^2\) belongs to \(L^\infty(\mathbb{T}^n)\).
		\end{enumerate}
		Suppose that there exists a rational trigonometric polynomial \(S\) satisfying the following conditions:
		\begin{enumerate}[(i)]
			\item \(S\in L^\infty(\mathbb{T}^n)\) is nonnegative, continuous at the origin, and \(S(0) = 1\).
			\item For almost all \(\xi \in \sigma(\phi)\),
			\begin{equation}
				S(\lambda \xi) \tau(\xi) \overline{\tau(\xi + \gamma)} + \sum_{l=1}^L q_l(\xi)\overline{q_l(\xi+\gamma)} = \begin{cases}
					S(\xi), & \text{if } \gamma = 0,\\
					0, & \text{if } \gamma\in \Gamma \setminus \{0\}.
				\end{cases}
				\label{eqn : OEP_condition}
			\end{equation}
		\end{enumerate}
		For each \(1 \leq l \leq L\), if we define \(\psi^{(l)} \in L^2(\mathbb{R}^n)\) via \(\widehat{\psi^{(l)}}(\lambda \cdot) = q_l(\cdot) \widehat{\phi}(\cdot)\), then the wavelet system \(X(\psi^{(1)}, \cdots, \psi^{(L)})\) is a TWF for \(L^2(\mathbb{R}^n)\).
		\label{result : OEP}
	\end{result}
	
	\begin{remark}
		According to \cite{BHan}, when \(\tau\) is given as a trigonometric polynomial with \(\tau(0)=1\) and satisfies the UEP condition (\ref{eqn : UEP_condition}) with appropriate \(q_l\) for \(1 \leq l \leq L\), the distribution \(\phi_\tau\), defined by \(\widehat{\phi}_\tau(\xi):= \prod_{j=1}^\infty \tau(\lambda^{-j}\xi)\), is a compactly supported refinable function in \(L^2(\mathbb{R}^n)\), where \(\tau\) serves as its refinement mask. 
		
		\label{remark : refinable function with TP}
	\end{remark}
	
	\begin{remark}
		In Section~\ref{section : refinable ftn with rationalTP}, we will demonstrate that when \(\tau\) is given as a rational trigonometric polynomial with \(\tau(0)=1\) and meeting the OEP condition~(\ref{eqn : OEP_condition}) with suitable~\(S\) and \(q_l,1 \leq l \leq L\) (as detailed in Proposition~\ref{prop : refinable ftn from rationalTP} and the subsequent remark), the distribution \(\phi_\tau\), defined as in Remark~\ref{remark : refinable function with TP}, is a refinable function in \(L^2(\mathbb{R}^n)\) with \(\tau\) as its refinement mask.
		This result can be regarded as a generalization of work done in Lemma 2.1 of \cite{BHan}, which is the case of trigonometric polynomial~\(\tau\).
		\label{remark : refinable function with rTP}
	\end{remark}
	
	We will use the aforementioned \(\phi_\tau\), derived from a given~\(\tau\), as the refinable function throughout the rest of this paper.
	
	\begin{definition}
		\begin{enumerate}[(a)]
			\item
			We refer to \(q_l\), where \(1 \leq l \leq L\), as a \textit{wavelet mask} if it satisfies the conditions specified either in Result~\ref{result : UEP} or Result~\ref{result : OEP}. 
			\item In addition, we refer to \((\tau, q_1, \cdots, q_L)\), where \(\tau\) is the refinement mask and \(q_l\) for \(1 \leq l \leq L\) are the  wavelet masks, as the \textit{combined MRA mask}.
			\item The \textit{wavelet system associated with this combined MRA mask}, denoted as \(X(\tau, q_1, \cdots, q_L)\), is defined as the wavelet system \(X(\psi^{(1)}, \cdots, \psi^{(L)})\), where \(\widehat{\psi^{(l)}}(\lambda\cdot) = q_l(\cdot) \widehat{\phi}_\tau(\cdot)\), \(1 \leq l \leq L\). 
		\end{enumerate}
	\end{definition}

	\subsection{Related works}
	\label{subsection : SOS-based construction}

	For a nonnegative (rational) trigonometric polynomial \(f\), we call that \(f\) \textit{has a sum of squares (SOS) representation} 
	if there exist (rational) trigonometric polynomials \(g_1, \cdots, g_L\) such that 
	\begin{equation*}
		f(\xi) = \sum_{l=1}^L \lvert g_l(\xi) \rvert^2, \quad \xi \in\mathbb{T}^n.
	\end{equation*}
	If such representation exists, we call \(g_l, 1 \leq l \leq L\), SOS generators.
	
	
	As briefly explained in the introduction, the specific form of wavelet masks that generate a TWF can be found through the SOS-based construction methods. There are two construction methods based on the SOS representation that we are interested in this paper.
	The first method assumes the sub-QMF condition and employs the unitary extension principle. For any given trigonometric polynomial~\(\tau\) satisfying \(\tau(0)=1\), the \textit{sub-QMF condition} is defined by
	\begin{equation}
		f_\tau(\xi):= 1- \sum_{\gamma\in\Gamma} \left| \tau \left( \xi + \gamma\right) \right|^2 \geq 0, \quad \xi\in\mathbb{T}^n.
		\label{eqn : sub-QMF}
	\end{equation}
	The second method works under the oblique sub-QMF condition by utilizing the oblique extension principle. For any given rational trigonometric polynomial \(\tau\) satisfying \(\tau(0) = 1\), and vanishing-moment recovery function~\(S\) that is a rational trigonometric polynomial, the \textit{oblique sub-QMF condition} is defined 
	by
	\begin{equation}
		f_{\tau,S}(\xi):=\frac{1}{S(\lambda \xi)} - \sum_{\gamma \in \Gamma} \frac{|\tau(\xi + \gamma)|^2}{S(\xi + \gamma)} \geq 0
		\label{eqn : oblique sub-QMF}
	\end{equation}
	for all \(\xi \in \mathbb{T}^n\) at which the rational trigonometric polynomial~\(f_{\tau,S}\) is defined.
	
	By finding a sum of squares representation of~\(f_\tau\) or \(f_{\tau,S}\), each construction method provides an exact form of combined MRA mask that constitutes a TWF.
	The detailed statements are given below respectively. Throughout the paper, we use the notation \(\Lambda\) to denote the set \(\lbrace0, \cdots, \lambda-1 \rbrace^n\), where \(\lambda\) is the dilation factor. The following result corresponds to the construction method derived from the SOS representation of \(f_\tau\),  specifically formulated for our setting.

	\begin{result}[\cite{BHan,LaiStockler}]
		Let \(\tau\) be a trigonometric polynomial satisfying \(\tau(0)=~1\).
		Suppose
		\begin{equation}
			f_\tau(\xi) = \sum_{l=1}^L |g_l(\lambda\xi)|^2, \quad \xi \in \mathbb{T}^n,
            \label{eqn : sub-QMF SOSv2}
		\end{equation}
		for some trigonometric polynomials \(g_1, \cdots, g_L\). For each \(\nu \in \Lambda\), define \(\tau_\nu\) by 
		\begin{equation}
			\tau_\nu(\xi) = \lambda^n \sum_{k \in \mathbb{Z}^n} h(\lambda k - \nu) e^{-i k \cdot \xi}, \quad \xi \in \mathbb{T}^n 
			\label{eqn : polyphase component}
		\end{equation}
		where \(\tau(\xi) = \sum_{k \in \mathbb{Z}^n} h(k)e^{-ik\cdot \xi}\).
		We define \(L+\lambda^n\) functions as follows:
		\begin{eqnarray*}
			q_{1,l}(\xi) & = & \tau(\xi) \overline{g_l(\lambda \xi)}, \quad l = 1, \cdots, L, \\
			q_{2,\nu}(\xi) & = & \lambda^{-n/2} \left( e^{i \nu \cdot \xi} - \tau(\xi) \overline{\tau_\nu (\lambda \xi)} \right), \quad \nu \in \Lambda.
		\end{eqnarray*}
		Then \((\tau, (q_{1,l}: l = 1, \cdots, L), (q_{2,\nu}: \nu \in \Lambda))\) satisfies the UEP condition~(\ref{eqn : UEP_condition}), and
		the wavelet system \(X(\tau, (q_{1,l}: l = 1, \cdots, L), (q_{2,\nu}: \nu \in \Lambda))\) is a compactly supported TWF for~\(L^2(\mathbb{R}^n)\).
		\label{result : LaiStockler subQMF}
	\end{result}
	
	\begin{remark*}
		We recall that the function \(\tau_\nu\) in (\ref{eqn : polyphase component}) is referred to as the \textit{polyphase component} of~\(\tau(\xi)=\sum_{k \in \mathbb{Z}^n} h(k)e^{-ik\cdot \xi}\), associated with \(\nu\in \Lambda\). In this case, they satisfy 
		\begin{equation*}
			\tau(\xi) = \lambda^{-n} \sum_{\nu\in\Lambda} \tau_\nu (\lambda \xi) e^{i\nu \cdot \xi}.
			\label{eqn : polyphase expression}
		\end{equation*}
	\end{remark*}
	
	The following result corresponds to the SOS-based construction assuming the oblique sub-QMF condition, though this particular formulation of the result is not explicitly presented. It is derived by integrating Theorems 1 and 2 of~\cite{[HLO]MultiTWFhighVM}, along with our Proposition~\ref{prop : refinable ftn from rationalTP} in Section~\ref{section : refinable ftn with rationalTP}. The original result in \cite{[HLO]MultiTWFhighVM} addresses the initiation of construction with a given refinable function and its associated refinement mask. However, for consistency, we have chosen to write the result beginning with a refinement mask, employing our Proposition~\ref{prop : refinable ftn from rationalTP} --- this proposition is used in the following result only for this --- rather than starting with a refinable function.

	\begin{result}[\cite{[HLO]MultiTWFhighVM}]
		Let \(\tau\) be a trigonometric polynomial satisfying \(\tau(0)=1\),
		and \(S\) be a rational trigonometric polynomial which is nonnegative on~\(\mathbb{T}^n\). Assume \(S\) is continuous at \(0\) with \(S(0) = 1\), and \(S, 1/S \in L^\infty (\mathbb{T}^n)\). Suppose  there are rational trigonometric polynomials \(g_l, 1 \leq l \leq L\), satisfying
		\begin{equation}
			f_{\tau,S}(\xi) = \sum_{l=1}^L |g_l(\lambda \xi)|^2, \quad \xi \in \mathbb{T}^n.
    		\label{eqn : oblique sub-QMF SOSv2}
		\end{equation}
		Furthermore, let~\(s_k, 1 \leq k \leq K\), be rational trigonometric polynomials with \(1/S(\xi) = \sum_{k=1}^K \lvert s_k(\xi) \rvert^2\) for all \(\xi \in \mathbb{T}^n\).
		We define \(L+\lambda^nK\) functions as
		\begin{align*}
			q_{1,l}(\xi) & = S(\lambda \xi) \tau(\xi) \overline{g_l(\lambda \xi)}, \quad 1 \leq l \leq L, \\
			q_{2,m}(\xi) & = S(\xi) a_m(\xi) - S(\lambda \xi) \tau(\xi) \sum_{\gamma \in \Gamma} \overline{\tau(\xi + \gamma)} a_m(\xi + \gamma), \quad 1 \leq m \leq \lambda^n K.
		\end{align*}
		Here, \(a_m(\xi) := \lambda^{-n/2} e^{i \nu_{l_m} \cdot \xi} s_{k_m}(\xi)\), where for each fixed \(m\), the integer \(k_m\) with \(1 \leq k_m \leq K\) is determined by \(m=\alpha_m K + k_m\) for some \(\alpha_m \in \mathbb{Z}\), and \(\nu_{l_m}\in \Lambda=:\{\nu_1, \cdots, \nu_{\lambda^n}\}\) is determined by \(m=\beta_m \lambda^n + l_m\) for some \(\beta_m \in \mathbb{Z}\) and \(1\le l_m\le \lambda^n\).
		
		Then \((\tau, (q_{1,l}: l = 1, \cdots, L), (q_{2,m}: m=1, \cdots,\lambda^nK))\) satisfies the OEP condition~(\ref{eqn : OEP_condition}), and the wavelet system \(X(\tau, (q_{1,l}: l = 1, \cdots, L), (q_{2,m}: m=1, \cdots,\lambda^nK))\) is a TWF for \(L^2(\mathbb{R}^n)\).
		
		\label{result : HLO oblique subQMF}
	\end{result}
	
	We recall some important quantities associated with a trigonometric polynomial \(\kappa\).
	\begin{definition}
		Let \(\kappa\) be a trigonometric polynomial. The \textit{accuracy number} of \(\kappa\) is defined to be the minimum order of the zeros of \(\kappa\) at \(\Gamma \setminus \{0\}\). The \textit{flatness number} of \(\kappa\) is the minimum order of the zeros of \(\kappa - 1\) at~\(0\). The \textit{vanishing moments} of \(\kappa\) is the minimum order of the zeros that \(\kappa\) has at \(0\).
	\end{definition}
	When \(\kappa\) is a rational trigonometric polynomial, we similarly define the quantities in the above definition, being careful to ensure that \( \kappa \) is well-defined at all relevant points.
	
	\medskip
	
	We now present some results about the vanishing moments of wavelet masks constructed in Results~\ref{result : LaiStockler subQMF} and~\ref{result : HLO oblique subQMF}, respectively. The following corresponds to the vanishing moments result about the construction under the sub-QMF condition. 
	\begin{result}[\cite{[HurZach]Interpolatory}]
		Assume the settings of Result~\ref{result : LaiStockler subQMF}. Suppose that \(\tau\) has accuracy number \(a \geq 1\), and flatness number \(b \geq 1\). Then, for each \(1 \leq l \leq L\), the wavelet mask \(q_{1,l}\) has exactly as many vanishing moments as \(g_l\), and the wavelet masks \(q_{2,\nu}, \nu\in \Lambda\), have at least \(\min\{a,b\}\) vanishing moments.
		\label{result : vanishing result of subQMF}
	\end{result}
	
	The following is the vanishing moments result concerning the construction under the oblique sub-QMF condition, which can be easily obtained from~\cite{[HLO]MultiTWFhighVM,[HurZach]Interpolatory}.
    
	
	\begin{result}[\cite{[HLO]MultiTWFhighVM,[HurZach]Interpolatory}]
		Assume the settings of Result~\ref{result : HLO oblique subQMF}. Suppose that \(\tau\) has accuracy number \(a\geq 1\) and \(f_{\tau,S}\) has vanishing moments \(c\geq 1\).
		Then, for each \(1 \leq l \leq L\), the wavelet mask \(q_{1,l}\) has exactly as many vanishing moments as~\(g_l\), and the wavelet masks \(q_{2,m}, 1 \leq m \leq \lambda^n K\), have at least \(\lfloor \min\{a, c/2\} \rfloor\) vanishing moments.
		\label{result : vanishing result of oblique subQMF}
	\end{result}
	
	\section{Construction of tight wavelet frames sharing responsibility}
	\label{section : our new construction}
	

	
	
	In the SOS-based construction methods outlined in Section~\ref{subsection : SOS-based construction}, it is crucial to identify (rational) trigonometric polynomials that yield a sum of squares representation for the function \(f_\tau\) of the sub-QMF condition, or \(f_{\tau,S}\) of the oblique sub-QMF condition. However, this task is nontrivial, especially in cases beyond one dimension.
	
	As each construction method begins with a given refinement mask, the entire challenge lies in determining the SOS generators, which are essential for the subsequent construction of wavelet masks. Hence, instead of initiating the whole process with a refinement mask, we propose simple conditions for both (rational) trigonometric polynomials, which can later serve as building blocks for refinement and wavelet masks. This approach distributes the burden more evenly between a refinement mask and the wavelet masks rather than placing the entire burden on the wavelet masks.
	
	The main results of this paper are detailed in the subsections that follow. In Section~\ref{subsection : our TWF using sub-QMF}, we demonstrate our new construction method, effective under the sub-QMF condition with a trigonometric polynomial refinement mask. On the other hand, in Section~\ref{subsection : our TWF using oblique subQMF}, we describe our construction method adapted for the oblique sub-QMF condition, applicable when employing a rational trigonometric polynomial as the refinement mask.
	
	
	\subsection{New tight wavelet frames with trigonometric polynomial refinement masks}
	
	\label{subsection : our TWF using sub-QMF}
	
	The following theorem demonstrates our new construction method that corresponds to the construction assuming the sub-QMF condition. Before presenting the statement, we note that, for simplicity, we restrict our construction method to the case where the resulting refinement mask is a trigonometric polynomial. However, this can be extended to the case of a rational trigonometric polynomial, as the unitary extension principle remains valid in such settings.

	\begin{theorem}
        Let \(N \leq \lambda^n\). For each \(1 \leq l \leq N\), suppose that there exist \(r_l \in \mathbb{N}\) and trigonometric polynomials \(p_l\) and \(g_l^{(r)}\), \(1 \leq r \leq r_l\), satisfying
        \begin{equation}
        	1 -  \left| p_l(\xi)\right|^2 = \lambda^n \sum_{r=1}^{r_l}\lvert g_l^{(r)}(\xi)\rvert^2, \quad \xi \in \mathbb{T}^n,
        	\label{eqn : our condition on p,g for subQMF}
        \end{equation}
        with \(p_l(0)=1\). In case of \(N<\lambda^n\), we set \(p_l\) a trigonometric polynomial satisfying \(\lvert p_l \rvert^2 = 1\) for \(l=N+1, \cdots, \lambda^n\).
        
        We define a trigonometric polynomial~\(\tau\) as
        \[
        \tau(\xi) = \lambda^{-n} \sum_{l=1}^{\lambda^n} p_l(\lambda \xi) e^{i \nu_l \cdot \xi},
        \]
        where \(\nu_1, \cdots, \nu_{\lambda^n}\) are distinct elements of \(\Lambda\), and define \((\sum_{l=1}^N r_l)+\lambda^n\) functions as follows:
        \begin{eqnarray*}
        	q_{1,l}^{(r)}(\xi) & = & \tau(\xi) \overline{g_l^{(r)}(\lambda \xi)}, \quad l = 1, \cdots, N,\,\, r = 1, \cdots, r_l,\\
        	q_{2,m}(\xi) & = & \lambda^{-n/2} \left( e^{i \nu_m \cdot \xi} - \tau(\xi) \overline{p_m(\lambda \xi)}\right), \quad m = 1, \cdots, \lambda^n.
        \end{eqnarray*}
        Then the wavelet system 
        \[
        X\left(\tau, (q_{1,l}^{(r)} : l = 1, \cdots, N, \; r = 1, \cdots, r_l), (q_{2,m} : m = 1, \cdots, \lambda^n)\right)
        \]
        is a compactly supported TWF for \(L^2(\mathbb{R}^n)\).
		\label{theorem: our subQMF construction}
	\end{theorem}

	\begin{proof}
		Since \(1- |p_l(\xi)|^2 = \lambda^n \sum_{r=1}^{r_l}|g_{l}^{(r)}(\xi)|^2\) holds for every \(1 \leq l \leq N\), we see that 
		\begin{equation*}
			\sum_{l=1}^N \sum_{r=1}^{r_l}|g_l^{(r)}(\lambda\xi)|^2 = \lambda^{-n} \sum_{l=1}^{\lambda^n} \left(1- |p_l(\lambda\xi)|^2 \right) = 1- \lambda^{-n} \sum_{l=1}^{\lambda^n} |p_l(\lambda\xi)|^2.
			\label{eqn : sum of our sub-QMF}
		\end{equation*}
		Furthermore, by the definition of \(\tau\), we also see that 
		\begin{eqnarray*}
			\sum_{\gamma \in \Gamma} \left| \tau \left(\xi + \gamma\right) \right|^2 & = & \lambda^{-2n} \sum_{l=1}^{\lambda^n} \sum_{\tilde{l}=1}^{\lambda^n} p_l(\lambda\xi) \overline{p_{\tilde{l}}(\lambda\xi)} e^{i(\nu_l - \nu_{\tilde{l}})\cdot \xi} \sum_{\gamma \in \Gamma} e^{i (\nu_l - \nu_{\tilde{l}}) \cdot \gamma} \\
			& = & \lambda^{-n}\sum_{l=1}^{\lambda^n} |p_l(\lambda\xi)|^2.
		\end{eqnarray*}
		Here, we use the \(2\pi\)-periodicity of trigonometric polynomials \(p_l\) and \(p_{\tilde{l}}\), and the following property: 
		\begin{equation*}
			\sum_{\gamma \in \Gamma} e^{i \nu \cdot \gamma} = 
			\begin{cases} 
				\lambda^n, & \text{if } \nu = 0, \\
				0, & \text{if }\nu \in\{\mu-\tilde{\mu}: \mu, \tilde{\mu}\in \Lambda\}  \setminus \{0\}.
			\end{cases}
		\end{equation*}
		Thus, we have \(1-\sum_{\gamma \in \Gamma} \left| \tau(\xi + \gamma)\right|^2 = \sum_{l=1}^N \sum_{r=1}^{r_l} |g_{l}^{(r)}(\lambda\xi)|^2\), and this completes the proof by invoking Result~\ref{result : LaiStockler subQMF}.
	\end{proof}
    
	We observe the functions \(p_l\), \(l=1,\cdots,\lambda^n\), in the above theorem take the role of polyphase components \(\tau_{\nu_l}\), \(l=1,\cdots,\lambda^n\), as~\(p_l\) are used to define \(\tau\) by
	\[\tau(\xi) = \lambda^{-n} \sum_{l=1}^{\lambda^n} p_l(\lambda \xi) e^{i \nu_l \cdot \xi},\] 
	whereas \(\tau_{\nu_l}\) satisfy 
	\[\tau(\xi) = \lambda^{-n} \sum_{l=1}^{\lambda^n} \tau_{\nu_l}(\lambda \xi) e^{i \nu_l \cdot \xi}\] 
	for any refinement mask \(\tau\) (see the remark right after Result~\ref{result : LaiStockler subQMF}).

    While constructing examples in \cite[Chapter 4]{CPSS}, the following expression, which is also a central component of our result, is employed (presented here in our normalization). For a given refinement mask \(\tau\) with corresponding polyphase components \(\tau_{\nu_l}, 1 \leq l \leq \lambda^n\), we have
    \begin{equation*}
        f_{\tau}(\xi) = 1 - \sum_{\gamma \in \Gamma} \lvert \tau(\xi + \gamma) \rvert^2 = 1 - \lambda^{-n} \sum_{l=1}^{\lambda^n} \lvert \tau_{\nu_l}(\lambda \xi) \rvert^2 = \lambda^{-n} \sum_{l=1}^{\lambda^n} \left( 1 - \lvert \tau_{\nu_l}(\lambda \xi)\rvert^2\right).
    \end{equation*}
    To utilize the UEP construction based on a sum-of-squares (SOS) representation, it is essential to verify the sub-QMF condition \(f_\tau \geq 0\). In \cite{CPSS}, it is noted that this verification can be done by finding an SOS representation for each polynomial \(1- \lvert \tau_{\nu_l}(\lambda\xi)\rvert^2, 1 \leq l \leq \lambda^n\), assuming all of them are nonnegative--an approach that aligns with our framework.

    However, a key distinction in our work lies in how the refinement mask \(\tau\) is constructed. In the method of \cite{CPSS}, once \(\tau\) is specified, each polyphase component \(\tau_{\nu_l}\) is determined by a corresponding exponential \(e^{i\nu_l \cdot \xi}\). In contrast, our approach begins with the building blocks \(p_l\), leaving us free to select the exponentials \(e^{i\nu_l \cdot \xi}\) independently, thus providing greater flexibility in designing \(\tau\). In the example section below, we illustrate how this design flexibility can be utilized, though within the scope of a relatively simple example.

    Moreover, while \cite{CPSS} mentions that it suffices to find an SOS representation for each polynomial \(1 - |\tau_{\nu_l}(\lambda \xi)|^2\), the examples therein primarily rely, in practice, on checking nonnegativity to verify the sub-QMF condition, rather than explicitly constructing such representations. In contrast, our approach aims to construct SOS representations for all these polynomials. Furthermore, since the polynomials in \cite{CPSS} are determined by the given refinement mask \(\tau\), the burden of finding SOS representations falls entirely on the other parts. To alleviate this, we introduce building blocks from which both refinement and wavelet masks are subsequently derived. This allows the necessary conditions to be distributed more evenly across the building blocks, rather than being concentrated on a single type of component.

	We now present a result about the vanishing moments of wavelet masks constructed in Theorem~\ref{theorem: our subQMF construction}. This is straightforward from Result~\ref{result : vanishing result of subQMF}.
	\begin{corollary}
		Assume the settings of Theorem~\ref{theorem: our subQMF construction}. Suppose that \(\tau\) has accuracy number \(a \geq 1\), and flatness number \(b\geq 1\). Then, for each \(1 \leq l \leq N\) and \(1 \leq r \leq r_l\), the wavelet mask \(q_{1,l}^{(r)}\) has exactly as many vanishing moments as \(g_l^{(r)}\), and the wavelet masks \(q_{2,m}, 1 \leq m \leq \lambda^n\), have at least \(\min\{a,b\}\) vanishing moments.
		\label{corollary : vanishing result our subQMF}
	\end{corollary}

    \subsubsection{Examples}
    \label{subsubsection: examples of our sub-QMF method}

    In this subsection, we present some examples constructed according to Theorem~\ref{theorem: our subQMF construction}. 
    
    Before presenting examples, we provide necessary and sufficient conditions under which the wavelet masks \(q_{2,m}, 1 \leq m \leq \lambda^n\), constructed in Theorem~\ref{theorem: our subQMF construction}, possess two vanishing moments. While Corollary~\ref{corollary : vanishing result our subQMF} presents a simple criterion for verifying the vanishing moments of wavelet masks \(q_{2,m}\),  it only yields a lower bound. Since our construction makes the refinement mask \(\tau\) directly from the building blocks~\(p_l\), it may be advantageous, whenever possible, to analyze the vanishing moments directly in terms of the \(p_l\)'s, rather than through the resulting refinement mask \(\tau\). The precise statement of the result is given below, while its proof is deferred to~\ref{section: appendix}. We also provide a specific example (see Example~\ref{example: Theorem 3.1 vanishing 2})
    by considering the conditions on \(p_l\) and \(\nu_l\) established below.

    \begin{proposition}
        Let \(\alpha\in\mathbb{N}_0^n\) be a multi-index with \(\lvert \alpha \rvert =1\). Then \(D^{\alpha} q_{2,m}(0)= 0\) for all \(1 \leq m \leq \lambda^n\) if and only if 
        \begin{equation*}
            A\begin{bmatrix}
                D^{\alpha}p_1(0)\\
                \vdots\\
                D^{\alpha}p_{\lambda^n}(0)
            \end{bmatrix} = \left(-\frac{i}{\lambda} \right) A \begin{bmatrix}
                \nu_1^\alpha\\
                \vdots\\
                \nu_{\lambda^n}^\alpha
            \end{bmatrix},\quad\hbox{where}\quad A:=\begin{bmatrix}
            1 & 0 & \cdots & 0 & -1\\
            0 & 1 & \cdots & 0 & -1\\
            \vdots & \vdots & \ddots & \vdots & \vdots\\
            0 & 0 & \cdots & 1 & -1
        \end{bmatrix}
        \end{equation*}
        is of size \((\lambda^n-1)\times\lambda^n\).
        \label{prop: vanishing moment for q2}
    \end{proposition}


    We now present our examples constructed using Theorem~\ref{theorem: our subQMF construction}. In particular, when \(r_l=1\) for all \(1 \leq l \leq N\), we denote \(g_l:=g_l^{(r)}\) and \(q_{1,l}:=q_{1,l}^{(r)}\). We begin by constructing an example based on Proposition~\ref{prop: vanishing moment for q2}.
    
    \begin{example}[Construction based on Proposition~\ref{prop: vanishing moment for q2}]
        \label{example: Theorem 3.1 vanishing 2}
        Let \(\lambda = 2\), \(n = 2\) and \(N=4\). For every \(1 \leq l \leq 4\), we set 
        \begin{equation*}
            p_l(\xi):= (1 + e^{-i(a_l \xi_1 + b_l\xi_2)})/2,\quad \hbox{for some \(a_l, b_l \in \mathbb{R}\)}.
        \end{equation*}
        Then, for \(1 \leq l \leq 4\), we get
        \(D^{(1,0)}p_l(0) = -(a_l/2)i \) and \(D^{(0,1)}p_l(0) = -(b_l/2)i.\)
        By setting \(\nu_1 = (1,1), \nu_2 = (1,0), \nu_3 = (0,1),\) and \(\nu_4 = (0,0)\), it suffices to find pairs \(\{(a_l,b_l):1 \le l \leq 4\}\) that satisfy the following conditions in order to ensure that the corresponding wavelet masks \(q_{2,m}, 1 \leq m \leq 4\), have two vanishing moments:
       \begin{align*}
            \left\{
            \begin{aligned}
                & D^{(1,0)}p_1(0) - D^{(1,0)}p_4(0) = -i(a_1 - a_4)/2 = -i/2,\\
                & D^{(1,0)}p_2(0) - D^{(1,0)}p_4(0) = -i(a_2 - a_4)/2 = -i/2,\\
                & D^{(1,0)}p_3(0) - D^{(1,0)}p_4(0) = -i(a_3 - a_4)/2 = 0,
            \end{aligned}
            \right.
            \\
            \left\{
            \begin{aligned}
                & D^{(0,1)}p_1(0) - D^{(0,1)}p_4(0) = -i(b_1 - b_4)/2 = -i/2,\\
                & D^{(0,1)}p_2(0) - D^{(0,1)}p_4(0) = -i(b_2 - b_4)/2 = 0,\\
                & D^{(0,1)}p_3(0) - D^{(0,1)}p_4(0) = -i(b_3 - b_4)/2 = -i/2.
            \end{aligned}
            \right.
        \end{align*}
        It is easy to see that each pair is determined once \((a_4,b_4)\) is fixed. By choosing \((a_4,b_4)=(1,0)\), we obtain \((a_1,b_1)=(2,1), (a_2,b_2)=(2,0)\) and \((a_3,b_3)=(1,1)\). This determines our \(p_l\)'s as follows:
        \begin{align*}
            &p_1(\xi) = (1+e^{-i(2\xi_1 + \xi_2)})/2,  \quad p_2(\xi) = (1+e^{-2i\xi_1})/2,\\
            &p_3(\xi) = (1+e^{-i(\xi_1 + \xi_2)})/2, \quad p_4(\xi) = (1+e^{-i\xi_1})/2.
        \end{align*}
          Then we see that \(\tau(\xi)\) equals to 
        \[\frac{1}{4} \left( \frac{e^{i(\xi_1 + \xi_2)}+ e^{i(-3\xi_1 - \xi_2)}}{2} + \frac{e^{i\xi_1}+ e^{-3i\xi_1}}{2} \right.
            \left. + \frac{e^{i\xi_2}+ e^{i(-2\xi_1 - \xi_2)}}{2} + \frac{1 + e^{-2i\xi_1}}{2}\right)\]
        and the corresponding \(g_l\)'s are given as follows:
        \begin{align*}
            &g_1(\xi) = (1-e^{-i(2\xi_1 + \xi_2)})/2,  \quad g_2(\xi) = (1-e^{-2i\xi_1})/2,\\
            &g_3(\xi) = (1-e^{-i(\xi_1 + \xi_2)}/2, \quad g_4(\xi) = (1-e^{-i\xi_1}/2.
        \end{align*}
        We note that the wavelet mask \(q_{2,m}, 1\leq m \leq 4\), has two vanishing moments by Proposition~\ref{prop: vanishing moment for q2}. Also, for each \(1 \leq l \leq 4\), the wavelet mask \(q_{1,l}\) has one vanishing moment since \(g_l\) has one vanishing moment. 

        This example underscores that the accuracy and flatness numbers of \(\tau\) may not always be appropriate indicators of the vanishing moments of \(q_{2,m}\) within our construction of Theorem~\ref{theorem: our subQMF construction}. Indeed, the flatness of \(\tau\) is \(1\), which implies that the wavelet masks \(q_{2,m}, 1 \leq m \leq 4\), have at least one vanishing moment from Corollary~\ref{corollary : vanishing result our subQMF}. However, from the above discussion, we know that the exact vanishing moments of \(q_{2,m}\)'s are in fact two. \qed
    \end{example}
	
	\begin{example}[Example from \cite{[Hur]TWFPD} in our construction setup]
		For \(1 \leq l \leq N \leq \lambda^n\), we consider an example in \cite{[Hur]TWFPD} by taking \(g_l(\xi)= \lambda^{-n/2} ((1-e^{-i\omega_l \cdot \xi})/2)^{m_l}\), where \(m_l\in \mathbb{N}\) is the vanishing moment,
		and \(\omega_l \in \mathbb{Z}^n\) is the direction vector with the initial point zero. 
		
		Let \(\lambda=2, n=2\) and \(N=3\) as in Example 4.1 of \cite{[Hur]TWFPD}. The vanishing moment \(m_l, 1 \leq l \leq N\), is all chosen by \(1\), and the direction vectors are chosen by 
		\begin{equation*}
			\begin{bmatrix}
				\omega_1 & \omega_2 & \omega_3
			\end{bmatrix} = 
			\begin{bmatrix}
				1 & 0 & 1\\
				0 & 1 & 1
			\end{bmatrix}.
		\end{equation*}
		
		Then for each \(1 \leq l \leq 3\), \(g_l\) is \((1-e^{-i \omega_l \cdot \xi})/4\), and \(p_l\) can be found as \((1+e^{-i \omega_l \cdot \xi})/2\), which satisfies the condition~(\ref{eqn : our condition on p,g for subQMF}), from Fej\'{e}r-Riesz lemma (see Lemma~6.1.3 in \cite{[Daubechies]TenLectures}). Furthermore, we set \(p_4 \equiv 1\). By setting \(\nu_l = \omega_l\) for \(1 \leq l \leq 3\) and \(\nu_4=0\), we have \(\tau\) as
		\begin{equation*}
			\tau(\xi) = \frac{1}{4} \left( \frac{e^{i\xi_1}+e^{-i\xi_1}}{2} + \frac{e^{i \xi_2} + e^{-i\xi_2}}{2} + \frac{e^{i(\xi_1+\xi_2)} + e^{-i(\xi_1 + \xi_2)}}{2} + 1\right).
		\end{equation*}
		As demonstrated in Example 4.1 of \cite{[Hur]TWFPD}, both the accuracy and flatness numbers of \(\tau\) are \(2\), indicating that \(q_{2,m}, 1\leq m \leq 4,\) have at least two vanishing moments. Additionally, for each \(1 \leq l \leq 3\), the wavelet mask \(q_{1,l}\) has exactly one vanishing moment since \(g_l\) has one vanishing moment. Thus we get a compactly supported TWF for \(\lambda=3\) and \(n=2\), consisting of~\(7\) mother wavelets with at least \(1\) vanishing moment. \qed
		
		\label{example:thm3_1_directional}
	\end{example}
	
	%
	%
	
	Moreover, we can identify a function \(p\) that pairs with a generalized form~\(g\) of the function \(g_l\), satisfying condition (\ref{eqn : our condition on p,g for subQMF}) by applying the Fej\'{e}r-Riesz lemma, rather than using the specific pair \((p_l, g_l)\) in the above Example~\ref{example:thm3_1_directional}. We defer its proof to~\ref{section: appendix}.
	
	
	
	
	\begin{lemma}
		Let \(g\) be a trigonometric polynomial on \(\mathbb{T}^n\), defined by \(g(\xi) = G(\xi \cdot \omega)\), where \(G\) is a univariate trigonometric polynomial on \(\mathbb{T}\) with \(G(0) =~0\), and \(\omega\)  is a fixed element of \(\mathbb{Z}^n\). If \(1-\lambda^n \lvert g(\xi) \rvert \geq 0\) for all \(\xi \in \mathbb{T}^n\), then there is a trigonometric polynomial \(p(\xi) = P(\xi \cdot \omega)\) on \(\mathbb{T}^n\) for some univariate trigonometric polynomial \(P\) on \(\mathbb{T}\), satisfying \(\lvert p(\xi) \rvert^2 = 1 - \lambda^n \lvert g(\xi) \rvert^2\) for all \(\xi \in \mathbb{T}^n\).
		
		\label{lemma : FejerRieszlemma}
	\end{lemma}
	

	We present further examples below. The following example illustrates the construction of a TWF using Theorem~\ref{theorem: our subQMF construction} for the case of \(\lambda = 3\).
	
	\begin{example}[Example with \(\lambda = 3\)]
		Let \(\lambda=3, n=1\) and \(N=2\). We consider
		\begin{equation*}
			\tau(\xi) = (3+4\cos \xi + 2 \cos 2\xi)/9,
		\end{equation*}
		which is used in Example 2 of \cite{[HurZach]Interpolatory}.
		From this \(\tau\), we can figure out \(p_l\) satisfying \(\tau(\xi) = (1/3) \sum_{l=1}^3 p_l(3 \xi) e^{i\nu_l \xi}\) for \(\nu_1=1\), \(\nu_2=2\), and \(\nu_3=0\) as follows:
		\begin{equation*}
			p_1(\xi) = (2 + e^{-i\xi})/3, \quad p_2(\xi) = (1 + 2e^{-i\xi})/3, \quad p_3(\xi)\equiv 1.
		\end{equation*}
		We see that \(\lvert p_1(\xi)\rvert^2 = \lvert p_2(\xi)\rvert^2=(1+8\cos^2(\xi/2))/9\). Then we need to find \(g_1\) and~\(g_2\) such that \(\lvert g_1 (\xi)\rvert^2 = \lvert g_2 (\xi)\rvert^2 = (8/27) \sin^2(\xi/2)\) to satisfy the equation~(\ref{eqn : our condition on p,g for subQMF}) together with given \(p_1\) and \(p_2\). Thus, we have 
		\begin{equation*}
			g_1(\xi) = g_2(\xi) = (\sqrt{6}/9) (1-e^{i\xi}).
		\end{equation*}
		We note that both the accuracy and flatness numbers of \(\tau\) are \(2\). Therefore, the wavelet masks \(q_{2,m}, 1 \leq m \leq 3,\) have at least two vanishing moments. Also, for each \(1 \leq l \leq 2\), the wavelet mask \(q_{1,l}\) has one vanishing moment since \(g_l\) has one vanishing moment. Thus we get a new univariate, compactly supported TWF for \(\lambda=3\), consisting of \(5\) mother wavelets with at least \(1\) vanishing moment.
		
		With the same \(\lambda, n\) and \(N\), we now consider 
		\begin{equation*}
			\tau(\xi) = (81 + 120 \cos \xi + 60 \cos 2\xi - 10 \cos 4\xi -8 \cos 5 \xi)/3^5,
		\end{equation*}
		which is appeared in Example 3 in \cite{[HurZach]Interpolatory}. We find \(p_l, 1 \leq l \leq 3,\) as  
		\begin{equation*}
			p_1(\xi) = \frac{-4 e^{-2i\xi} +30e^{-i\xi} +60 -5 e^{i\xi}}{3^4}, \,\, p_2(\xi) = \frac{-5 e^{-2i\xi} +60e^{-i\xi} +30 -4e^{i\xi}}{3^4},
		\end{equation*} 
		and \(p_3 \equiv 1\).
		For \(l=1\) and~\(2\), we see that 
		\begin{equation*}
			\lvert p_l(\xi) \rvert^2 = (961 + 12480 \cos^2(\xi/2) - 8160 \cos^4 (\xi/2) + 1280 \cos^6 (\xi/2))/81^2.
		\end{equation*}
		Then for \(l=1\) and \(2\),  the function \(3^9 \lvert g_l (\xi) \rvert^2\) equals to
		\begin{eqnarray*}
			 && 12480 \sin^2(\xi/2) -8160 \left(2 \cos^2 (\xi/2) \sin^2 (\xi/2) +\sin^4 (\xi/2)\right) \\
			& + &1280 \left(3 \cos^4(\xi/2) \sin^2 (\xi/2) + 3 \cos^2(\xi/2) \sin^4(\xi/2) + \sin^6(\xi/2)\right)\\
			& =   &\sin^4 (\xi/2) \left( 5600 - 1280 \cos^2 (\xi/2)\right).
		\end{eqnarray*}
		Thus for \(l=1\) and \(2\), we set 
		\[g_l(\xi)=\frac{\sqrt{30}}{476} \left( - e^{-i\xi} + 2 - e^{i\xi}\right) \left( 3 \sqrt{3} + \sqrt{35} + (3\sqrt{3} - \sqrt{35}) e^{i\xi} \right).\]
		As demonstrated in Example 3 of \cite{[HurZach]Interpolatory}, both the accuracy and flatness numbers of \(\tau\) are \(4\). Thus, wavelet masks \(q_{2,m}, 1 \leq m \leq 3,\) have at least \(4\) vanishing moments. Also, for each \(1 \leq l \leq 2\), the wavelet mask \(q_{1,l}\) has two vanishing moments since \(g_l\) has two vanishing moments in the above argument. Thus, we get a new univariate, compactly supported TWF for \(\lambda=3\), consisting of \(5\) mother wavelets with at least \(2\) vanishing moments. \qed
	\end{example}
	
	The next example is a bivariate (i.e., \(n = 2\)) TWF constructed by applying Theorem~\ref{theorem: our subQMF construction}.
	
    \begin{example}[Example involving two auxiliary functions \(g_l^{(1)}\), \(g_l^{(2)}\)]
		Let \(\lambda= 2, n=2\) and \(N=4\). Then for each \(1 \leq l \leq 4\), we will find \(p_l\) and \(g_l^{(1)}, g_l^{(2)}\) satisfying (\ref{eqn : our condition on p,g for subQMF}) by simply employing 
        \begin{equation*}
            1- \cos^4 t = 4 \left( (\sin^2 t \cdot \cos^2 t)/2 + (\sin^4 t)/4\right)
        \end{equation*}
        (for example, \(\lvert p_l (\xi_1, \xi_2)\rvert^2 = \cos^4(\xi_1)\), \(\lvert g_l^{(1)} (\xi_1, \xi_2)\rvert^2 = (\sin(\xi_1)\cdot \cos(\xi_1)/\sqrt{2})^2\), and \(\lvert g_l^{(2)} (\xi_1, \xi_2)\rvert^2 = (\sin^2 (\xi_1)/2)^2\)).
		We choose, for \(l=1 \text{ and }2\),
		\begin{equation*}
			p_l(\xi_1, \xi_2) = \frac{e^{2i\xi_1} + 2 + e^{-2i\xi_1}}{4} \quad \text{and} \quad \begin{cases}
			    g_l^{(1)}(\xi_1, \xi_2) = 
                (e^{2i\xi_1} - e^{-2i\xi_1})/4\sqrt{2},\\
                g_l^{(2)}(\xi_1, \xi_2) = 
                (e^{2i\xi_1}-2+e^{-2i\xi_1})/8,
			\end{cases}
		\end{equation*}
		\begin{equation*}
			p_l(\xi_1, \xi_2) = \frac{e^{2i\xi_2} + 2 + e^{-2i\xi_2}}{4} \quad \text{and} \quad \begin{cases}
			    g_l^{(1)}(\xi_1, \xi_2) = 
                (e^{2i\xi_2} - e^{-2i\xi_2})/4\sqrt{2},\\
                g_l^{(2)}(\xi_1, \xi_2) = (e^{2i\xi_2}-2+e^{-2i\xi_2})/8,
			\end{cases}
		\end{equation*}
        for \(l = 3\) and \(4\). Also, by choosing \(\nu_1 = (0,0), \nu_2 = (1,0), \nu_3 = (0,1)\) and \(\nu_4 = (1,1)\), we have
		\begin{eqnarray*}
			\tau(\xi) & = & \frac{1}{16} \left( 
            e^{4i\xi_1} + 2 + e^{-4i\xi_1} + e^{5i\xi_1} + 2 e^{i\xi_1} + e^{-3i\xi_1}\right.\\
            & & \left.\qquad + e^{5i\xi_2} + 2 e^{i\xi_2} + e^{-3i\xi_2} + e^{i(\xi_1 + 5\xi_2)} + 2 e^{i(\xi_1 + \xi_2)} + e^{i(\xi_1 -3\xi_2)}\right).
		\end{eqnarray*}
		The accuracy and the flatness of \(\tau\) are both equal to \(1\). Thus, the wavelet masks \(q_{2,m}, 1 \leq m \leq 4,\) have at least one vanishing moment. For each \(1 \leq l \leq 4\), the wavelet mask \(q_{1,l}^{(1)}\) and \(q_{1,l}^{(2)}\) have one and two vanishing moments respectively, since \(g_l^{(1)}\) and \(g_l^{(2)}\) has one and two vanishing moments respectively. Thus, we get a new bivariate, compactly supported TWF for \(\lambda=2\), consisting of \(8\) mother wavelets with at least \(1\) vanishing moment. \qed

	\end{example}
    
	
	\subsection{New tight wavelet frames with rational trigonometric polynomial refinement masks}
	\label{subsection : our TWF using oblique subQMF}
	
	In this subsection, we present our new method that corresponds to the construction assuming the oblique sub-QMF condition. For simplicity, we focus on the case where each \(p_l\) is associated with a single \(g_l\) for all \(1 \leq l \leq N\), though the method naturally extends to the general case with multiple \(g_l\)'s per \(p_l\), as described in Theorem~\ref{theorem: our subQMF construction}. The following theorem is a generalization of our previous construction outlined in Theorem~\ref{theorem: our subQMF construction}. More specifically, it precisely corresponds to our earlier construction if \(S\) is set to~\(1\). The complete statement is provided below.
	%
	
	\begin{theorem}
		Let \(N \leq \lambda^n\). Let \(S\) be a rational trigonometric polynomial, which is nonnegative on \(\mathbb{T}^n\). Assume that \(S\) is continuous at~\(0\) with \(S(0) = 1\), and \(S, 1/S \in L^\infty (\mathbb{T}^n)\). For each \(1 \leq l \leq N\), suppose that rational trigonometric polynomials \(p_l\) and \(g_l\) satisfy
		\begin{equation}
			\frac{1}{S(\xi)} - |p_l(\xi)|^2 = \lambda^n |g_l(\xi)|^2
			\label{eqn : our condition on p,g for oblique subQMF}
		\end{equation} 
		for all \(\xi\in\mathbb{T}^n\) at which both sides are defined, and \(p_l(0)=1\).
		Let \(s\) be a rational trigonometric polynomial satifying \(\lvert s \rvert^2=1/S\) with \(s(0) = 1\). 
		In case of \(N<\lambda^n\), let \(p_l\) be a rational trigonometric polynomial satisfying \(\lvert p_l \rvert^2 = 1/S\) for \(l = N+1, \cdots,\lambda^n\).
		We define a rational trigonometric polynomial \(\tau\) as
		\begin{equation*}
			\tau(\xi) = \left. \left(\lambda^{-n}\sum_{l=1}^{\lambda^n} p_l(\lambda \xi) e^{i \nu_l \cdot \xi}\right)\middle/\left({\overline{s(\xi)}}\right)\right.
		\end{equation*}
		where \(\nu_1, \cdots, \nu_{\lambda^n}\) are distinct elements of \(\Lambda\),
		and \(N+\lambda^n\) functions as follows:
		\begin{eqnarray}
			q_{1,l}(\xi) & = & S(\lambda \xi) \tau(\xi) \overline{g_l(\lambda \xi)}, \quad 1 \leq l \leq N, \label{eq:q1} \\
			q_{2,m}(\xi) & = & \lambda^{-\frac{n}{2}} \left(S(\xi) s(\xi) e^{i\nu_m \cdot \xi} - S(\lambda \xi) \tau(\xi) \overline{p_m(\lambda \xi)}\right), 1 \leq m \leq \lambda^n.\label{eq:q2}
		\end{eqnarray}
		Then the wavelet system \(X(\tau, (q_{1,l}:l=1, \cdots, N), (q_{2,m} : m=1, \cdots, \lambda^n))\) is a TWF for~\(L^2(\mathbb{R}^n)\).
		\label{theorem : our oblique subQMF construction}
	\end{theorem}

	We employ the following lemma, which is obtained by slightly modifying the settings in Theorem 1 of \cite{[HLO]MultiTWFhighVM}, to prove the above theorem. This lemma provides a specific form of \(q_l, 1 \leq l \leq L\), satisfying the OEP condition (\ref{eqn : OEP_condition}), together with \(\tau\) and \(S\), if \(f_{\tau,S}\) in (\ref{eqn : oblique sub-QMF}) has an SOS representation with SOS generators. In the original formulation presented in~\cite{[HLO]MultiTWFhighVM}, \(\tau\) is considered as a trigonometric polynomial. However, one can easily see that the proof is still valid when \(\tau\) is a rational trigonometric polynomial, following the same arguments as those used in \cite{[HLO]MultiTWFhighVM}. Therefore, for our analysis, we consider \(\tau\) to be a rational trigonometric polynomial.
	
	\begin{lemma}
		Let \(S\) be a nonzero rational trigonometric polynomial that is nonnegative on \(\mathbb{T}^n\), and \(\tau\)\ be a rational trigonometric polynomial satisfying \(\tau(0)=1\). Also, let \(s\) be a rational trigonometric polynomial that satisfy \(\lvert s \rvert^2 = 1/S\).
		If \(f_{\tau,S}\) in (\ref{eqn : oblique sub-QMF}) has an SOS representation, i.e., there are rational trigonometric polynomials \(g_{\tilde{l}}, 1 \leq l \leq \tilde{L}\), satisfying		\begin{equation*}			
			f_{\tau,S}(\xi) = \sum_{\tilde{l}=1}^{\tilde{L}} \lvert g_{\tilde{l}}(\lambda \xi) \rvert^2
		\end{equation*}		
		for all \(\xi\in \mathbb{T}^n\) at which both sides are defined,		
		then there exist rational trigonometric polynomials \(\{q_l\}_{l=1}^L\) such that for all \(\gamma \in \Gamma\) and \(\xi \in \mathbb{T}^n\) at which both sides are defined:
		\begin{equation*}			
			S(\lambda \xi) \tau(\xi) \overline{\tau(\xi + \gamma)} + \sum_{l=1}^L q_l(\xi) \overline{q_l(\xi+\gamma)} = \begin{cases}				
				S(\xi), & \text{if }\gamma = 0,\\				
				0, & \text{otherwise.}			
			\end{cases}		
		\end{equation*}		
		Specifically, \(\{q_l\}_{l=1}^L:= \{q_{1,{\tilde{l}}}\}_{\tilde{l}=1}^{\tilde{L}} \bigcup \{q_{2,m}\}_{m=1}^{\lambda^n}\) where 		\begin{align*}			
			q_{1,{\tilde{l}}}(\xi) & = S(\lambda \xi) \tau(\xi) \overline{g_{\tilde{l}}(\lambda \xi)}, \quad 1 \leq \tilde{l} \leq {\tilde{L}}, \\			
			q_{2,m}(\xi) & = S(\xi) a_m(\xi) - S(\lambda \xi) \tau(\xi) \sum_{\gamma \in \Gamma} \overline{\tau(\xi + \gamma)} a_m(\xi + \gamma), \quad 1 \leq m \leq \lambda^n.		
		\end{align*}
		Here, \(a_m(\xi) := \lambda^{-n/2} e^{i \nu_m \cdot \xi} s(\xi)\) where \(\nu_1, \cdots, \nu_{\lambda^n}\) are distinct elements in \(\Lambda\).
		\label{lemma : HLO Theorem 1}
	\end{lemma}

	Now, we present our proof of Theorem~\ref{theorem : our oblique subQMF construction}. In this proof, we utilize a few propositions presented in Section~\ref{section : refinable ftn with rationalTP}, employing \(\phi_\tau\) as our refinable function (see Proposition~\ref{prop : refinable ftn from rationalTP}) and making use of its convergence property at \(0\) (see Proposition~\ref{prop : convergence of prod at 0}).
	
	\begin{proof}[Proof of Theorem~\ref{theorem : our oblique subQMF construction}]
		Without loss of generality, we assume that \(N<\lambda^n\). Since the equation (\ref{eqn : our condition on p,g for oblique subQMF}) holds for every \(1 \leq l \leq N\), and since we set \(p_l\) to be satisfied \(\lvert p_l \rvert^2 = 1/S\) for \(N+1 \leq l \leq \lambda^n\) , we see that 
		\begin{equation*}
			\sum_{l=1}^N |g_l(\xi)|^2 = \lambda^{-n} \sum_{l=1}^{\lambda^n} \left( \frac{1}{S(\xi)} - |p_l(\xi)|^2 \right) = \frac{1}{S(\xi)} - \lambda^{-n} \sum_{l=1}^{\lambda^n} |p_l(\xi)|^2.
		\end{equation*}
		From the definition of \(\tau\), we first see that
		\begin{eqnarray*}
			\frac{|\tau(\xi)|^2}{S(\xi)} & = & \lambda^{-2n} \left(\sum_{l=1}^{\lambda^n} p_l(\lambda \xi)e^{i\nu_l \cdot \xi} \right) \left(\sum_{\tilde{l} = 1}^{\lambda^n} \overline{p_{\tilde{l}}(\lambda \xi)} e^{-i \nu_{\tilde{l}} \cdot \xi} \right) \\
			& = &\lambda^{-2n} \sum_{l=1}^{\lambda^n} \sum_{\tilde{l} = 1}^{\lambda^n} p_l(\lambda \xi) \overline{p_{\tilde{l}}(\lambda \xi)} e^{i \xi \cdot (\nu_l - \nu_{\tilde{l}})}.
		\end{eqnarray*}
		Then we have
		\begin{eqnarray*}
			\sum_{\gamma \in \Gamma} \frac{|\tau(\xi+\gamma)|^2}{S(\xi+\gamma)} & = & \lambda^{-2n} \sum_{l=1}^{\lambda^n} \sum_{\tilde{l} = 1}^{\lambda^n} p_l(\lambda \xi) \overline{p_{\tilde{l}}(\lambda \xi)} e^{i \xi \cdot (\nu_l - \nu_{\tilde{l}})} \left( \sum_{\gamma \in \Gamma} e^{i \gamma \cdot (\nu_l - \nu_{\tilde{l}})}\right) \\
			& = & \lambda^{-n} \sum_{l=1}^{\lambda^n} |p_l(\lambda \xi)|^2.
		\end{eqnarray*}
		Thus, we have \(1/S(\lambda\xi) - \sum_{\gamma \in \Gamma} |\tau(\xi+\gamma)|^2/S(\xi+\gamma) = \sum_{l=1}^N |g_l(\lambda\xi)|^2\), which shows that \(f_{\tau,S}\) in (\ref{eqn : oblique sub-QMF}) has an SOS representation with SOS generators \(g_l, 1 \leq l \leq~N\). Then by applying Lemma~\ref{lemma : HLO Theorem 1}, there exist rational trigonometric polynomials \(q_k, 1 \leq k \leq N+\lambda^n\), satisfying 
		\begin{equation}
			S(\lambda \xi) \tau(\xi) \overline{\tau(\xi + \gamma)} + \sum_{k=1}^{N+\lambda^n} q_k(\xi) \overline{q_k(\xi+\gamma)} = \begin{cases}
				S(\xi), & \text{if }\gamma = 0,\\
				0, & \gamma \in \Gamma \setminus \{0\}.
			\end{cases}
			\label{eqn : our OEP condition}
		\end{equation}
		Here, we observe that the functions \(\{q_{1,l}\}_{l=1}^N\) in (\ref{eq:q1}) and \(\{q_{2,m}\}_{m=1}^{\lambda^n}\) in (\ref{eq:q2}) constitute \(\{q_k\}_{k=1}^{N+\lambda^n}\).
		
		We now utilize the OEP construction method presented in Result~\ref{result : OEP} to complete our proof. Note that we can employ \(\phi_\tau\in L^2(\mathbb{R}^n)\) as our refinable function employing Proposition~\ref{prop : refinable ftn from rationalTP} under our settings. We check each of the conditions in Result~\ref{result : OEP}. Clearly, conditions (i) and~(ii) there hold based on the given assumptions. We see that the condition~(b) holds straightforwardly from Proposition~\ref{prop : convergence of prod at 0}. Furthermore, one can easily see that the condition~(c) holds using the same method as done in the proof of Theorem~2 of~\cite{[HLO]MultiTWFhighVM}. 
  
  For the condition (a) of Result~\ref{result : OEP}, we first observe that for each \(1 \leq l \leq N\), \(p_l\) is essentially bounded from (\ref{eqn : our condition on p,g for oblique subQMF}), since \(1/S\in L^\infty(\mathbb{T}^n)\). Combining this with the fact that \(1/\overline{s}\) is essentially bounded, which is from \(S\in L^\infty(\mathbb{T}^n)\), we conclude that \(\tau\in L^\infty(\mathbb{T}^n)\).
		Examining~(\ref{eqn : our OEP condition}) with \(\gamma=0\), we observe that \(\sum_{l=1}^{N+\lambda^n} \lvert q_l(\xi) \rvert^2 = S(\xi) - S(\lambda \xi) \lvert \tau(\xi) \rvert^2\). From this and from the facts that \(S, \tau \in L^\infty(\mathbb{T}^n)\), we conclude that \(q_l \in L^\infty(\mathbb{T}^n)\) for every \(1\leq l \leq N+\lambda^n\), establishing the condition (a).
	\end{proof}

	In terms of the settings, Theorem~\ref{theorem : our oblique subQMF construction} differs from Result~\ref{result : HLO oblique subQMF} in two key aspects. Firstly, the refinement mask \(\tau\) is constructed as a rational trigonometric polynomial. Secondly, the variable \(K\), representing the number of SOS generators of \(1/S\), in Result~\ref{result : HLO oblique subQMF}, is chosen as \(1\) here.
	
	In the approach of Theorem~\ref{theorem : our oblique subQMF construction}, as in Theorem~\ref{theorem: our subQMF construction}, an important aspect is that we avoid the need to solve the SOS problem of finding \(g_l, 1 \leq l \leq N\), which satisfies 
	\begin{equation*}
		\frac{1}{S(\lambda \xi)} - \sum_{\gamma \in \Gamma} \frac{|\tau(\xi + \gamma)|^2}{S(\xi + \gamma)} = \sum_{l=1}^N |g_l(\lambda\xi)|^2
	\end{equation*}
	for a given refinement mask \(\tau\), under the given conditions. These conditions we impose distribute the responsibility between \(p_l\) and \(g_l\), preventing the entire burden from falling solely on finding \(g_l\).
	
	We now show a result about the vanishing moments of wavelet masks constructed in Theorem~\ref{theorem : our oblique subQMF construction}. 
	\begin{corollary}
		Assume the settings of Theorem~\ref{theorem : our oblique subQMF construction}. Suppose that \(\tau\) has accuracy number \(a\geq 1\), and for each \(1 \leq l \leq N\), \(g_l\) has vanishing moments~\(d_l\). 
		Then, for each \(1 \leq l \leq N\), the wavelet mask \(q_{1,l}\) has exactly \(d_l\) vanishing moments, and the wavelet masks \(q_{2,m}, 1 \leq m \leq \lambda^n,\) have at least \(\min\{a, d_1, \cdots, d_N\}\) vanishing moments.
		\label{corollary : vanishing result of our oblique subQMF}
	\end{corollary}
	
	
	\begin{proof}
		We first observe that Result~\ref{result : vanishing result of oblique subQMF} remains valid when \(\tau\) is a rational trigonometric polynomial, not just a trigonometric polynomial. Using this generalized version of Result~\ref{result : vanishing result of oblique subQMF}, it is straightforward to see that the vanishing moments of \(q_{1,l}, 1 \leq l \leq N,\) are as stated above. For wavelet masks \(q_{2,m}, 1\leq m \leq \lambda^n,\) we already know that \(f_{\tau,S}(\xi) = \sum_{l=1}^N \lvert g_l(\lambda \xi) \rvert^2\) holds from the proof of Theorem~\ref{theorem : our oblique subQMF construction} under the setting of the theorem. We then obtain the stated result by applying the generalized version of Result~\ref{result : vanishing result of oblique subQMF} once more. 
	\end{proof}
	
	We present some examples that can be constructed via Theorem~\ref{theorem : our oblique subQMF construction}.
	
	\begin{example}
		Let \(n=1\) and \(N= \lambda\). We choose \(S\) by
		\begin{equation*}
			{1}/{S(\xi)} = \A + (1-\A)\cos^2 (\xi/2)
		\end{equation*}
		for any fixed real number \(\A > 1\). Since 
		\begin{equation*}
			\A + (1-\A) \cos^2 (\xi/2) + (1-\A) \sin^2 (\xi/2) = 1
		\end{equation*}
		holds, 
		we will find \(p_l=p\) and \(g_l=g\), for all \(1\leq l \leq  \lambda\), such that 
		\begin{equation*}
			\lvert g(\xi) \rvert^2 = \frac{\A-1}{\lambda} \sin^2 \frac{\xi}{2} \quad \text{and} \quad \lvert p(\xi) \rvert^2 = 1
		\end{equation*}
		to satisfy \(1/S(\xi) - \lambda \lvert g(\xi) \rvert^2 = \lvert p(\xi) \rvert^2\). Then we have
		\begin{equation*}
			g_l(\xi) = \sqrt{\frac{\A-1}{\lambda}} \left(\frac{1-e^{i\xi}}{2} \right) \quad \text{and} \quad p_l(\xi) \equiv 1, \quad 1 \leq l \leq  \lambda.
		\end{equation*}
		Also, we can find \(s\) satisfying \(1/S(\xi) = \lvert s(\xi) \rvert^2\) by 
		\begin{equation*}
			s(\xi) = ((1+\sqrt{\A}) + (1- \sqrt{\A})e^{i\xi})/2.
		\end{equation*}
		Thus, we have a refinement mask \(\tau\) as
		\begin{equation*}
			\tau(\xi) = \frac{2}{\lambda} \cdot\frac{1+e^{i\xi} + \cdots + e^{i(\lambda-1) \xi}}{ \left(1- \sqrt{\A}\right)e^{-i\xi}+(1+\sqrt{\A})}.
		\end{equation*}
		Since the accuracy number of \(\tau\) is \(1\), and all \(g_l, 1 \leq l \leq \lambda\), have one vanishing moment, we see that wavelet masks \(q_{1,l}, 1 \leq l \leq \lambda\), have exactly one vanishing moment, and \(q_{2,m}, 1 \leq m \leq \lambda\), have at least one vanishing moment from Corollary~\ref{corollary : vanishing result of our oblique subQMF}. 
		Thus, we get a univariate TWF, consisting of \(2\lambda\) mother wavelets with at least \(1\) vanishing moment.
		
		Consider the case when \(\lambda=2\) and \(\A= 5/3\). Then we can select \(S\) by \(3/(4-\cos\xi)\).
		We set \(g_1(\xi) = g_2(\xi) = (1/2\sqrt{3}) (1-e^{i\xi})\) and \(p_1(\xi) = p_2(\xi) \equiv 1\). Moreover, \(s\) is determined by \((3+\sqrt{15} + (3-\sqrt{15}) e^{i\xi})/6\). Then we have
		\begin{equation*}
			\tau(\xi) = \frac{3(1+e^{i\xi})}{(3-\sqrt{15}) e^{-i\xi}+3+\sqrt{15}}.
		\end{equation*}
		The accuracy number of \(\tau\) is 1 and each \(g_l, 1 \leq l \leq 2\), has one vanishing moment, hence the wavelet masks \(q_{1,l}, 1 \leq l \leq 2\), and \(q_{2,m}, 1 \leq m \leq 2,\) have at least one vanishing moment. As a result, we get a univariate TWF for \(\lambda=2\), consisting of \(4\) mother wavelets with at least \(1\) vanishing moment.
	\end{example}
	
	In fact, the construction described in the above example can be generalized to yield higher vanishing moments for \( g_l \), and consequently, for the wavelet masks \( q_{1,l} \), where \( 1 \leq l \leq \lambda \).
	
	\begin{example}
		Let \(n=1\) and \(N=\lambda\). For \(k=1,2,3,...\), we choose \(S\) by 
		\begin{equation*}
			\frac{1}{S(\xi)} = \A + (1-\A) \left(\left(\cos^2 \frac{\xi}{2}+\sin^2\frac{\xi}{2}\right)^k-\sin^{2k} \frac{\xi}{2} \right)
		\end{equation*}
		for any fixed real number \(\A>1\). 
		Since
		\begin{equation*}
			\A + (1-\A) \left(\left(\cos^2 \frac{\xi}{2}+\sin^2\frac{\xi}{2}\right)^k-\sin^{2k} \frac{\xi}{2} \right) + (1-\A) \left( \sin^{2k} \frac{\xi}{2}\right) = 1
		\end{equation*}
		holds, if we set \(p_l(\xi) \equiv 1\) for \( 1 \leq l \leq \lambda\), and 
		\begin{equation*}
			g_l(\xi) = E(k;\xi) \sqrt{\frac{\A-1}{\lambda}}  \left(\frac{1-e^{i\xi}}{2} \right)^k, \quad 1 \leq l \leq \lambda, 
		\end{equation*}	
		where \(E(k;\xi)\) is a trigonometric polynomial satisfying \(\lvert E(k;\xi) \rvert \equiv 1\), 	
		then \(1/S(\xi) - \lambda \lvert g_l(\xi) \rvert^2 = \lvert p_l(\xi) \rvert^2\), for all \(1\leq l \leq  \lambda\). 
		Since \(1/S\) is a trigonometric polynomial satisfying \(1/S \geq 1\) on \(\mathbb{T}\), by Fej\'{e}r-Riesz lemma, there exists a trigonometric polynomial \(s\) such that  \(1/S=\lvert s \rvert^2 \) on \(\mathbb{T}\), and we can proceed with this \(s\) to define the refinement mask \(\tau\) and the wavelet masks as in Theorem~\ref{theorem : our oblique subQMF construction}. In particular, we get 
		\begin{equation*}
			\tau(\xi) = (1+ e^{i\xi} + \cdots + e^{i(\lambda-1)\xi})/(\lambda{\overline{s(\xi)}}).
		\end{equation*}
		Note that the accuracy number of \(\tau\) is \(1\), and all \(g_l, 1 \leq l \leq \lambda\), have \(k\) vanishing moments. Hence the wavelet masks \(q_{1,l}, 1 \leq l \leq \lambda\), have exactly \(k\) vanishing moments, whereas \(q_{2,m}, 1 \leq m \leq \lambda\), have at least one vanishing moment from Corollary~\ref{corollary : vanishing result of our oblique subQMF}.
		
		The construction in the previous example corresponds to the case where \(k=1\) with \(E(1; \xi) \equiv 1\).
		
		When \(k=2\), the function \(1/S(\xi)\) is given as 
		\begin{equation*}
			\A + (1-\A) \left(\cos^4 \frac{\xi}{2} + 2 \sin^2\frac{\xi}{2} \cos^2 \frac{\xi}{2}\right) = \A + (1-\A)  \left(\frac{5 + 4 \cos \xi - \cos 2\xi}{8}\right),
		\end{equation*}
		and the function \(s\) such that  \(1/S=\lvert s \rvert^2 \) on \(\mathbb{T}\) can be found concretely as 
		\begin{equation*}
			s(\xi) = \frac{1 + \sqrt{\A} + \sqrt{2 + 2 \sqrt{\A}}}{4}e^{-i\xi} + \frac{1-\sqrt{\A}}{2} + \frac{1+ \sqrt{\A} - \sqrt{2+2\sqrt{\A}}}{4}e^{i\xi}.
		\end{equation*}		
		By choosing \(E(2; \xi)=e^{-i\xi}\) for symmetry, we get
		\begin{equation*}
			g_l(\xi) = \sqrt{\frac{\A-1}{\lambda}} \left(\frac{e^{-i\xi} - 2 + e^{i\xi}}{4}\right) \quad \text{and} \quad p_l(\xi) \equiv 1, \quad 1 \leq l \leq \lambda, 
		\end{equation*}		
		and the refinement mask \(\tau\) in this case becomes
		\begin{equation*}
			\frac{4}{\lambda} \cdot\frac{1+ e^{i\xi} + \cdots + e^{i(\lambda-1)\xi}}{(1 + \sqrt{\A} - \sqrt{2 + 2 \sqrt{\A}})e^{-i\xi} + 2(1-\sqrt{\A}) + (1+ \sqrt{\A} + \sqrt{2+2\sqrt{\A}})e^{i\xi}}.
		\end{equation*}
	\end{example}

	\section{On refinable functions with rational trigonometric polynomial refinement masks}
	\label{section : refinable ftn with rationalTP}
	
	As briefly mentioned in Remark~\ref{remark : refinable function with TP} following Result~\ref{result : OEP} within Section~\ref{subsection : TWF and EP}, a refinable function~\(\phi_\tau \in L^2(\mathbb{R}^n)\) can be derived from a trigonometric polynomial \(\tau\) satisfying the UEP condition~\cite{BHan}. In this section, we provide sufficient conditions for a refinable function derived from a rational trigonometric polynomial \(\tau\) to be in \(L^2(\mathbb{R}^n)\). Furthermore, we demonstrate that the refinable function resulting from such \(\tau\) satisfies the specific property required for constructing a TWF.

	The first proposition establishes that a refinable function can be defined from a rational trigonometric polynomial in the same way as it is derived from a trigonometric polynomial in \cite{BHan}, which can be considered as an extension of Lemma 2.1 in~\cite{BHan}. Our argument follows the methods employed in proving the lemma, and Theorem 2 in \cite{[HLO]MultiTWFhighVM}. This proposition is used in the proof of Theorem~\ref{theorem : our oblique subQMF construction} in Section~\ref{subsection : our TWF using oblique subQMF}.
	
	\begin{proposition}
		Let \(S\) be a nonzero rational trigonometric polynomial, which is nonnegative on \(\mathbb{T}^n\), and \(\tau\) be a rational trigonometric polynomial satisfying \(\tau(0) = 1\). 
		If the following conditions hold:
		\begin{enumerate}[(i)]
			\item the oblique sub-QMF condition: 
			\begin{equation*}
				f_{\tau,S}(\xi) = \frac{1}{S(\lambda \xi)} - \sum_{\gamma \in \Gamma} \frac{\lvert \tau(\xi + \gamma) \rvert^2}{S(\xi + \gamma)} \geq 0 
			\end{equation*}
			for all \(\xi \in \mathbb{T}^n\) where the rational trigonometric polynomial \(f_{\tau,S}\) is defined, and
			\item \(S\) is continuous at \(0\) with \(S(0)=1\), and \(1/S \in L^\infty(\mathbb{R}^n)\),
		\end{enumerate}
		then \(\phi_\tau\) defined by specifying its Fourier transform as \(\widehat{\phi_\tau}(\xi):=\prod_{j=1}^\infty \tau(\lambda^{-j} \xi)\) is a refinable function in \(L^2(\mathbb{R}^n)\).
		\label{prop : refinable ftn from rationalTP}
	\end{proposition}
	
	\begin{proof}
		Let \(f_0(\xi):= \chi_{[-\pi,\pi)^n}(\xi) S(\xi)^{-1/2}\), and \(f_m(\xi) := \tau(\lambda^{-1} \xi) f_{m-1} (\lambda^{-1} \xi)= \chi_{[-\lambda^m\pi,\lambda^m\pi)^n}(\xi) S(\lambda^{-m} \xi) ^{-1/2} \prod_{j=1}^m \tau(\lambda^{-j} \xi)\) for \(m \in \mathbb{N}\). We prove by induction that for all \(\xi \in \mathbb{T}^n\) at which both sides are defined, \(\lbrack f_m, f_m\rbrack(\xi ) := \sum_{k\in\mathbb{Z}^n} \lvert f_m(\xi+2\pi k) \rvert^2 \leq 1/S(\xi)\) holds for \(m \geq 0\).  Clearly, it holds in case of \(m=0\) since \(\lbrack f_0, f_0 \rbrack(\xi) = 1/S(\xi)\). We now suppose that \(\lbrack f_{m-1}, f_{m-1} \rbrack(\xi) \leq 1/S(\xi)\) for some \(m\in\mathbb{N}\). Then 
		\begin{eqnarray*}
			\lbrack f_m, f_m \rbrack (\xi) 
			& = & \sum_{k\in\mathbb{Z}^n} \lvert \tau(\lambda^{-1} (\xi+2\pi k))\rvert^2 \lvert f_{m-1} (\lambda^{-1} (\xi + 2\pi k)) \rvert^2 \\
			& = & \sum_{\gamma \in \Gamma} \sum_{k \in \mathbb{Z}^n} \lvert \tau(\lambda^{-1} \xi + \gamma + 2\pi k) \rvert^2 \lvert f_{m-1} (\lambda^{-1}\xi + \gamma + 2\pi k)\rvert^2 \\
			& = & \sum_{\gamma \in\Gamma} \lvert \tau(\lambda^{-1}\xi + \gamma) \rvert^2 \lbrack f_{m-1}, f_{m-1} \rbrack (\lambda^{-1} \xi + \gamma) \\
			& \leq & \sum_{\gamma \in\Gamma}  \frac{\lvert \tau(\lambda^{-1}\xi + \gamma) \rvert^2}{S(\lambda^{-1} \xi + \gamma)} \leq \frac{1}{S(\xi)},
		\end{eqnarray*}
		where the last inequality holds from the oblique sub-QMF condition.
		Then we obtain
		\begin{equation*}
			\lVert f_m(\xi) \rVert_{L^2(\mathbb{R}^n)}^2 = \int_{[-\pi, \pi)^n} \lbrack f_m, f_m \rbrack (\xi) d \xi \leq \int_{[-\pi, \pi)^n} \frac{1}{S(\xi)} d \xi \leq C (2\pi)^n
		\end{equation*}
		for some \(C>0\) since \(1/S \in L^\infty(\mathbb{T}^n)\). We also note that \(f_m(\xi) \to \widehat{\phi_\tau}(\xi)\) as \(m \to \infty\) by using the continuity of \(S\) at \(0\) and \(S(0)=1\). By Fatou's lemma, we see that
		\begin{equation*}
			\lVert \widehat{\phi_\tau}(\xi) \rVert_{L^2(\mathbb{R}^n)}^2 = \int \lim_{m \to \infty} \lvert f_m(\xi) \rvert^2 d\xi \leq \liminf_{m\to\infty} \int \lvert f_m (\xi) \rvert^2 d\xi \leq C(2\pi)^n<\infty.
		\end{equation*}
		Thus, we have \(\widehat{\phi_\tau} \in L^2(\mathbb{R}^n)\) which implies \(\phi_\tau \in L^2(\mathbb{R}^n)\). Finally, we have 
		\begin{equation*}
			\widehat{\phi_\tau}(\lambda \xi) = \prod_{j=1}^\infty \tau(\lambda^{1-j} \xi) = \tau(\xi) \prod_{j=1}^\infty \tau(\lambda^{-j} \xi) = \tau(\xi) \widehat{\phi_\tau}(\xi).
		\end{equation*}
		This completes the proof.
	\end{proof}
	
	\begin{remark*}
		We note that the condition (i) in the above proposition can be substituted with the OEP condition (\ref{eqn : OEP_condition}), as their equivalence has been demonstrated in Theorem~1 of \cite{[HLO]MultiTWFhighVM}: there exist rational trigonometric polynomials  \(q_{l}, 1 \leq l \leq L\) such that for all \(\xi \in \mathbb{T}^n\) at which both sides are defined:
		\begin{equation*}
			S(\lambda \xi) \tau(\xi) \overline{\tau(\xi + \gamma)} + \sum_{l=1}^L q_l(\xi) \overline{q_l(\xi + \gamma)} = \begin{cases}
				S(\xi), & \text{if } \gamma=0, \\
				0, & \text{if } \gamma \in \Gamma \setminus \{0\}.
			\end{cases}
		\end{equation*}
	\end{remark*}

	
		The following proposition is also used in the proof of Theorem~\ref{theorem : our oblique subQMF construction} when we employ Result~\ref{result : OEP}. Specifically, the condition (b) of Result~\ref{result : OEP} corresponds to the convergence of the refinable function at \(0\). In the proposition below, we establish the convergence of the product \(\prod_{j=1}^\infty \tau(\lambda^{-j} \xi)\), which later becomes the refinable function \(\phi_\tau\) under additional conditions as presented in Proposition~\ref{prop : refinable ftn from rationalTP}, at \(\xi=0\).

	\begin{proposition}
		Let \(\tau\) be a rational trigonometric polynomial satisfying \(\tau(0)=1\). Then 
		\begin{equation*}
			\lim_{\xi\to0} \prod_{j=1}^\infty \tau(\lambda^{-j} \xi) = 1.
		\end{equation*}
		\label{prop : convergence of prod at 0}
	\end{proposition}
	
	To demonstrate the above proposition, we employ the following fundamental result presented in \cite[p.141]{[stein]complex}, specialized to our setting.
	
	\begin{result}[\cite{[stein]complex}]
		Suppose \(\{F_j\}\) be a sequence of complex-valued functions on the open set \(K \subset \mathbb{R}^n\). If there exist constants \(c_j>0\) such that 
		\begin{equation*}
			\sum_{j=1}^\infty c_j < \infty \quad \text{and} \quad \lvert F_j(x) - 1 \rvert \leq c_j,\,\, \forall x \in K,
		\end{equation*}
		then the product \(\prod_{j=1}^\infty F_j(x)\) converges uniformly in \(K\).
		\label{lemma : stein product convergence}
	\end{result}
	
	\begin{proof}[Proof of Proposition~\ref{prop : convergence of prod at 0}]
		Since a rational trigonometric polynomial \(\tau= \tau_1/\tau_2\) has value as \(1\) at \(0\), we see that a trigonometric polynomial \(\tau_2\) is nonvanishing at \(0\). Since every trigonometric polynomial is continuous on \(\mathbb{T}^n\), we can find \(0<\epsilon<1\) such that \(\tau_2\) is nonvanishing on \(D:= \{\xi \in \mathbb{T}^n : \lvert \xi \rvert \leq \epsilon\}\). Thus we observe that \(\tau\) is a \(C^\infty\) function on~\(D\). Then, for all \(\xi \in D_{0}:= \{\xi \in \mathbb{T}^n : \lvert \xi \rvert <\epsilon/2\}\), we see that \(\lvert \tau(\lambda^{-j} \xi) - \tau(0) \rvert \leq \sup_{z \in D} \lvert \nabla\tau(z) \rvert \lvert\lambda^{-j} \xi \rvert\). This shows that \(\tau\) is Lipschitz continuous on \(D_0\).
		Then we have \(\lvert \tau(\lambda^{-j}\xi) - 1 \rvert = \lvert \tau(\lambda^{-j}\xi) - \tau(0) \rvert \leq C \lvert \lambda^{-j} \xi \rvert < C \lambda^{-j}\) for all \(\xi \in D_0\) with \(C := \sup_{z\in D} \lvert \nabla\tau (z)\rvert\).
		Moreover, we have \(\sum_{j=1}^\infty C \lambda^{-j} = C/(\lambda-1)<\infty\). Therefore, the product \(\prod_{j=1}^\infty \tau(\lambda^{-j} \xi)\) converges uniformly in \(D_0\) from Result~\ref{lemma : stein product convergence}. Finally, since \(\tau\) is continuous at~\(0\), we have
		\begin{equation*}
			\lim_{\xi\to0} \prod_{j=1}^\infty \tau(\lambda^{-j} \xi) = \prod_{j=1}^\infty \lim_{\xi\to 0} \tau(\lambda^{-j} \xi) = 1.
		\end{equation*}
	\end{proof}

	\section{Conclusion}
	\label{section : conclusion}
	
	This paper diverges from the traditional approach of constructing TWFs from a given refinable function. We introduced conditions for (rational) trigonometric polynomials, which later serve as building blocks of refinement and wavelet masks, distributing the responsibility between masks. Under these conditions, we presented the exact form of masks that generate a TWF (c.f., Theorems~\ref{theorem: our subQMF construction} and~\ref{theorem : our oblique subQMF construction}). Notably, we established sufficient conditions for the refinable function to be well-defined in \(L^2(\mathbb{R}^n)\) when the refinement mask is a rational trigonometric polynomial (c.f., Proposition~\ref{prop : refinable ftn from rationalTP}).

    \section*{Acknowledgement}

    This work was supported by the National Research Foundation of Korea~(NRF) [Grant Numbers 2021R1A2C1007598]. 

\appendix
\section{Proofs of auxiliary results in Section~\ref{subsubsection: examples of our sub-QMF method}}
\label{section: appendix}
\renewcommand{\thesection}{\Alph{section}} 
\makeatletter
\renewcommand\@seccntformat[1]{\appendixname\ \csname the#1\endcsname.\hspace{0.5em}}
\makeatother

    In this appendix, we provide the proofs of Proposition~\ref{prop: vanishing moment for q2} and Lemma~\ref{lemma : FejerRieszlemma}, both of which appear in Section~\ref{subsubsection: examples of our sub-QMF method}. We begin by proving Proposition~\ref{prop: vanishing moment for q2}. To prove this proposition, we make use of the following lemma, whose proof is straightforward and thus omitted.

    \begin{lemma}
        Let \(\alpha\in\mathbb{N}_0^n\) be a multi-index and \(\kappa\) be a trigonometric polynomial, that is,
        \begin{equation*}
            \kappa(\xi) = \sum_{k\in\mathbb{Z}^n} h(k) e^{-ik\cdot \xi}.
        \end{equation*} Then 
        \begin{equation*}
            D^\alpha \overline{\kappa(\xi)} = \overline{D^\alpha \kappa(\xi)}.
        \end{equation*}
        Moreover, if all coefficients of \(\kappa\) are real, then \(D^\alpha \overline{\kappa(0)} = (-1)^{\lvert \alpha \rvert} D^\alpha \kappa(0)\).
        \label{lemma:lemma for conjugate}
    \end{lemma}
    
    \begin{proof}[Proof of Proposition~\ref{prop: vanishing moment for q2}]

        Fix \(1\leq m \leq \lambda^n\). 
        For a multi-index \(\alpha\), we compute the derivative
        \begin{align*}
        D^{\alpha} q_{2,m}(\xi) 
        = \lambda^{-n/2} \big[
        (i\nu_{m}^\alpha) e^{i \nu_m \cdot \xi} 
        - \big( 
        D^{\alpha} \tau(\xi) \overline{p_m(\lambda \xi)} 
        + \lambda \tau(\xi) D^{\alpha} \overline{p_m(2\xi)}
        \big) \big].
        \end{align*}
        Evaluating at \(\xi = 0\) and applying Lemma~\ref{lemma:lemma for conjugate}, we obtain
        \begin{equation}
        D^{\alpha} q_{2,m}(0) 
        = \lambda^{-n/2} \left[
        (i\nu_{m}^\alpha) 
        - \left( D^{\alpha} \tau(0) 
        - \lambda D^{\alpha} p_m(0) \right)
        \right].
        \label{eqn:derivative-qm}
        \end{equation}
        By differentiating \(\tau\), we have
        \[
        D^{\alpha} \tau(\xi) = \lambda^{-n} \sum_{l=1}^{\lambda^n} 
        \lambda D^{\alpha} p_l(\lambda \xi)e^{i \nu_l \cdot \xi} 
        + (i\nu_l^\alpha) p_l(\lambda \xi)
         e^{i \nu_l \cdot \xi}, 
        \]
        which implies
        \[
        D^{\alpha} \tau(0) 
        = \lambda^{1-n} \sum_{l=1}^{\lambda^n} D^{\alpha} p_l(0) 
        + (i\lambda^{-n}) \sum_{l=1}^{\lambda^n} \nu_l^\alpha.
        \]
        Substituting this into \eqref{eqn:derivative-qm} yields that \(D^{\alpha}q_{2,m}(0)=0\) is equivalent to the following equation:
        \begin{equation*}
            \sum_{l=1}^{\lambda^n} D^{\alpha}p_l(0) - \lambda^n D^{\alpha} p_m(0) = -\frac{i}{\lambda}\left( \sum_{l=1}^{\lambda^n} \nu_{l}^\alpha - \lambda^n \nu_{m}^\alpha\right).
        \end{equation*}
        Thus, we have that \(D^{\alpha} q_{2,m}(0)=0\) for all \(1 \leq m \leq \lambda^n\) if and only if 
        \begin{align*}
            B \begin{bmatrix}
                D^{\alpha}p_1(0)\\
                \vdots\\
                D^{\alpha}p_{\lambda^n}(0)
            \end{bmatrix} = \left(-\frac{i}{\lambda} \right)
            B \begin{bmatrix}
                \nu_1^\alpha \\
                \vdots\\
                \nu_{\lambda^n}^\alpha
            \end{bmatrix}, \quad\text{where}\quad B:=\begin{bmatrix}
                1-\lambda^n & \cdots & 1\\
                \vdots &  \ddots & \vdots\\
                1 &\cdots & 1-\lambda^n
            \end{bmatrix}
        \end{align*}
        is of size \(\lambda^n\times\lambda^n\).
        Hence, we obtain the desired result by noting that \(B\) is row equivalent to the matrix formed by appending a zero row to \(A\).
    \end{proof}

    We now present the proof of Lemma~\ref{lemma : FejerRieszlemma}.
    
    \begin{proof}[Proof of Lemma~\ref{lemma : FejerRieszlemma}]
    		Our proof is a straightforward application of Fej\'{e}r-Riesz lemma~\cite[Lemma 6.1.3]{[Daubechies]TenLectures}. Since \(1-\lambda^n \lvert G(\mu) \rvert^2=1-\lambda^n \lvert g(\xi) \rvert^2 \geq 0\) with \(\mu:=\xi \cdot \omega\), there is a univariate trigonometric polynomial \(P\) such that \(1 - \lambda^n \lvert G(\mu) \rvert^2 = \lvert P (\mu) \rvert^2\). By setting \(p(\xi):= P(\mu)\), we see that \(\lvert p(\xi) \rvert^2 = 1- \lambda^n \lvert g(\xi) \rvert^2\),  \(\xi \in \mathbb{T}^n\), holds.
    \end{proof}


	
	
	
	
\end{document}